\documentclass[a4paper,11pt]{amsart}

\usepackage{comment}

\usepackage{etoolbox}
\ifdef{\crop}{%
\usepackage[includeheadfoot,twoside=False,paperwidth=448pt,paperheight=587pt,rmargin=15pt,lmargin=15pt,tmargin=15pt,bmargin=15pt]{geometry}%
}{%
\setlength{\topmargin}{22mm}
\addtolength{\topmargin}{-1in}
\setlength{\oddsidemargin}{27mm}
\addtolength{\oddsidemargin}{-1in}
\setlength{\evensidemargin}{27mm}
\addtolength{\evensidemargin}{-1in}
\setlength{\textwidth}{156mm}
\setlength{\textheight}{240mm}
}%

\usepackage{mathtools}
\usepackage{color}
\usepackage[active]{srcltx}
\usepackage{amsmath,amsthm,amsxtra}
\usepackage{amssymb}

\usepackage[mathscr]{eucal}

\usepackage{bm}
\usepackage[all]{xy}

\usepackage{aliascnt}

\usepackage{tabularx}
\newcolumntype{C}{>{\centering\arraybackslash}X} %中央揃え

\theoremstyle{plain}
\newtheorem{thm}{Theorem}[section]
\newtheorem*{thm*}{Theorem}

\newaliascnt{prop}{thm}
\newaliascnt{cor}{thm}
\newaliascnt{lem}{thm}
\newaliascnt{claim}{thm}
\newaliascnt{defn}{thm}
\newaliascnt{ques}{thm}
\newaliascnt{conj}{thm}
\newaliascnt{fact}{thm}
\newaliascnt{rem}{thm}
\newaliascnt{ex}{thm}
\newaliascnt{sett}{thm}
\newtheorem{prop}[prop]{Proposition}
\newtheorem{cor}[cor]{Corollary}
\newtheorem{lem}[lem]{Lemma}
\newtheorem{claim}[claim]{Claim}
\newtheorem*{prop*}{Proposition}
\newtheorem*{cor*}{Corollary}
\newtheorem*{lem*}{Lemma}
\newtheorem*{claim*}{Claim}
\theoremstyle{definition}
\newtheorem{defn}[defn]{Definition}

\newtheorem{sett}[sett]{Setting}
\newtheorem*{defn*}{Definition}
\newtheorem*{ques*}{Question}
\newtheorem*{conj*}{Conjecture}

\newtheorem*{prob*}{Problem}
\newtheorem{rem}[rem]{Remark}
\newtheorem{ex}[ex]{Example}
\newtheorem*{fact*}{Fact}
\newtheorem*{rem*}{Remark}
\newtheorem*{ex*}{Example}
\aliascntresetthe{prop}
\aliascntresetthe{cor}
\aliascntresetthe{lem}
\aliascntresetthe{claim}
\aliascntresetthe{defn}
\aliascntresetthe{ques}
\aliascntresetthe{conj}
\aliascntresetthe{fact}
\aliascntresetthe{rem}
\aliascntresetthe{ex}
\aliascntresetthe{sett}

\usepackage[pointedenum]{paralist}
\usepackage{varioref}
\labelformat{equation}{\textnormal{(#1)}}
\labelformat{enumi}{\textnormal{(#1)}}

\usepackage[a4paper]{hyperref}

\def\textsectionN~{\textsection{}}

\renewcommand\phi{\varphi}
\renewcommand\epsilon{\varepsilon}
\renewcommand\leq{\leqslant}
\renewcommand\geq{\geqslant}

\makeatletter
\newcommand{\set}{  \@ifstar{\@setstar}{\@set}}\newcommand{\@setstar}[2]{\{\, #1 \mid #2 \,\}}
\newcommand{\@set}[1]{\{ #1 \}}
\makeatother

\newcommand{\trans}[1][1]{\raisebox{#1ex}{\scriptsize\kern0.1em$t$\kern-0.1em}}

\DeclareMathOperator{\Hom}{Hom}

\DeclareMathOperator{\Pic}{Pic}

\DeclareMathOperator{\id}{id}

\DeclareMathOperator{\pr}{pr}

\def\Z{\mathbb{Z}}
\def\Q{\mathbb{Q}}
\def\R{\mathbb{R}}
\def\C{\mathbb{C}}

\def\r+{\mathbb{R}_{\geq 0}}

\def\ep{\varepsilon}

\def\r+{{\R}_{\geq 0}}
\def\q+{{\Q}_{\geq 0}}

\def\*c{\C^{\times}}

\def\<{\langle}
\def\>{\rangle}

\def\ni{\noindent}
\def\ra{\rightarrow}
\def\lra{\Leftrightarrow}

\def\C{\mathbb {C}}

\def\Q{\mathbb {Q}}
\def\R{\mathbb {R}}

\def\Z{\mathbb {Z}}

\newcommand{\calo}{\mathcal {O}}

\makeatletter
  
  \@addtoreset{equation}{section}
\makeatother

\title[Projective normality and basepoint-freeness thresholds of abelian varieties]{Projective normality and basepoint-freeness thresholds of general polarized abelian varieties}

\author[A.~Ito]{Atsushi~Ito}
\address{Department of Mathematics, Faculty of Natural Science and Technology,
Okayama University,
Okayama, Japan}
\email{ito-atsushi@okayama-u.ac.jp}

\begin{document}

\maketitle

\begin{abstract}
For a polarized abelian variety $(X,L)$, Z.~Jiang and G.~Pareschi introduce an invariant $\beta(X,L)$, called the basepoint-freeness threshold.
Using this invariant,
we show that a general polarized abelian variety $(X,L)$ of dimension $g$ is projectively normal
if $\chi(L) \geq 2^{2g-1}$ and the type of $L$ is not $(2,4,\dots,4)$.
This bound is sharp since it is known that any polarized abelian variety of type $(2,4,\dots,4)$ is not projectively normal.
We also give an application of $\beta(X,L)$ to the Infinitesimal Torelli Theorem for $Y \in |L|$.
\end{abstract}

\section{Introduction}

Throughout this paper, we work over the complex number field $\C$.

Let $(X,L) $ be a polarized abelian variety.
Then $L^{\otimes m}$ is projectively normal if $m \geq 3$  \cite{MR480543},
or $m=2$ and $L$ is basepoint free  (see \cite{MR966402} for a more precise version).
However, the projective normality of $L$ itself is more subtle.

In this paper, we study the projective normality of general $(X,L)$ of a given type.
Recall that for a polarized abelian variety $(X,L)$ of dimension $g$,
we can associate a sequence of positive integers $(d_1,\dots,d_g)$ with $d_1|d_2|\cdots|d_g$, called the \emph{type} of $L$. 
In \cite{MR2964474},
J-M.~Hwang and W-K.~To show that a general polarized abelian variety $(X,L)$ of a given type is projectively normal if $\chi(L) \geq (8g)^g/2g!$.

The purpose of this paper is to give a  sharp lower bound of $\chi(L)$ which implies projective normality of general $(X,L) $ as follows:

\begin{thm}\label{thm_proj_normality}
Let $(X,L)$ be a general polarized abelian variety of dimension $g$.
Then  $L$ is projectively normal if $\chi(L) \geq 2^{2g-1}$ and the type of $(X,L)$ is not $(2,4,\dots,4)$.

In particular,  $L$ is projectively normal if $\chi(L) > 2^{2g-1}$.
\end{thm}

We note that $(8g)^g/2g!  \sim (8e)^g/2\sqrt{2 \pi g} $ for $g \gg 1$ by Stirling's formula.
Hence the bound $ 2^{2g-1}=4^g/2$ replaces the base $8e$ of the exponent by $4$.
In fact,
a polarized abelian variety $(X,L)$ of type $(2,4,\dots,4) $ satisfies $\chi(L)=  2^{2g-1}$ and is never projectively normal by \cite{MR1638159}.
Hence the bound $\chi(L) > 2^{2g-1}$ in \autoref{thm_proj_normality} is sharp.

\vspace{2mm}

In the rest of  Introduction, we explain the strategy of the proof of \autoref{thm_proj_normality}.
Recently, 
Z.~Jiang and G.~Pareschi  \cite{MR4157109} introduce
an invariant $\beta(L)=\beta(X,L) \in (0,1]$,
called the \emph{basepoint-freeness threshold}, 
for 
a polarized abelian variety $(X,L)$,
and prove the following theorem.

\begin{thm}[{\cite[Theorem D, Corollary E]{MR4157109}}]\label{thm_Bpf_threshold}
Let $L$ be an ample line bundle on an abelian variety $X$.
Then 
\begin{enumerate}
\item $L$ is basepoint free if and only if $\beta(L) <1$.
\item $L$ is projectively normal if $\beta(L) <1/2$.
\end{enumerate}
\end{thm}

\autoref{thm_Bpf_threshold} is generalized to higher syzygies and jet ampleness by 
 \cite{MR4114062}, \cite{CaucciThesis}, \cite{Ito:2021aa}, and 
it turned out that 
$\beta(X,L)$ is quite useful to investigate the linear system of $L$.
Hence it is important to give upper bounds of $\beta(L)$ as in \cite{Ito:2020aa}, \cite{MR4395094}, \cite{Ito:2020ab}, \cite{MR4363833}.

In  \cite{Ito:2020ab},
we give upper bounds of $\beta(X,L)$ for general $(X,L)$ of type $(1,\dots,1,d)$.
In this paper, we give a generalization of the result to arbitrary type (see \autoref{thm_beta_general_type}).
Combining the bounds in \autoref{thm_beta_general_type} and an idea in \cite{MR2231163} which decomposes types to lower dimensional ones,
we can show $\beta(X,L) <1/2$ for $(X,L)$ in the statement of \autoref{thm_proj_normality}.
Hence \autoref{thm_proj_normality} follows from \autoref{thm_Bpf_threshold} (2).

\vspace{1mm}
We also give an application of $\beta(X,L)$ to the Infinitesimal Torelli Theorem  for hypersurfaces in abelian varieties. 
Recall that a smooth projective variety $Y$ satisfies \emph{the Infinitesimal Torelli Theorem} if the natural map
\begin{align*}
H^1(Y , T_Y) \rightarrow \Hom(H^0(Y,K_Y), H^1(V, \Omega^{\dim Y-1}_Y) )
\end{align*}
is injective.
P.~Blo{\ss} investigates the case $ Y \in |L| $ for a polarized abelian variety $(X,L)$
and shows that the Infinitesimal Torelli Theorem holds for any smooth $Y \in |L|$ if $X$ is simple and $h^0(L) > (g/(g-1))^{g} \cdot g!$ in  \cite[Theorem 1.1]{Bloss:2019aa}.
Under the assumption that $(X,L)$ is general,
we refine this result as follows:

\begin{thm}\label{cor_ITT}
Let $(X,L)$ be a general polarized abelian variety of dimension $g \geq 2$ and of type $(d_1,\dots,d_g)$.
If $d_1 + \dots +d_g \geq 2g $ and $(d_1,\dots,d_g)$ is not in the following list, then  the Infinitesimal Torelli Theorem holds for any smooth $Y \in |L|$.
\begin{itemize}
\item $g=2$ and $ (d_1,d_2) = (1,3), (1,4)$, %$(1,5) , (1,6), (2,2), (2,4)$, 
\item $g=3 $ and $ (d_1,d_2,d_3) = (1,1,4),(1,1,5), (1,1,6) , (1,3,3)$, 
\item $g \geq 4$ and $  (d_1,\dots,d_g) = (1,\dots,1,g+1) , (1,\dots,1,g+2), (1,\dots,1,2,g)$.
\end{itemize}
\end{thm}

\vspace{2mm}
This paper is organized as follows. 
In \autoref{sec_preliminary}, we recall some notation.
In \autoref{sec_construction}, we give examples of polarized abelian varieties whose types can be explicitly determined.
In \autoref{sec_any_type}, we prove \autoref{thm_beta_general_type}, which gives some upper bounds of $\beta(X, L)$ for general $(X,L)$.
In \autoref{sec_decomposition}, we study upper bounds of $\beta(X,L)$ by decomposing the type of $L$ to lower dimensional types
and prove \autoref{cor_ITT}.
In \autoref{sec_projective_normal}, we prove \autoref{thm_proj_normality}.

\subsection*{Acknowledgments}
The author was supported by JSPS KAKENHI Grant Numbers 17K14162, 21K03201.

\section{Preliminaries}\label{sec_preliminary}

By convention, we set $\sum_{i \in I} a_i =0 $ and $ \prod_{i \in I} a_i = 1$ when $I= \emptyset$.

\subsection{Abelian varieties}\label{subsec_type}

Let $X$ be an abelian variety of dimension $g$.
The origin of $X$ is denoted by $o_X$ or $o \in X$.
For $b \in \Z$, the multiplication-by-$b$ isogeny is denoted by
\[
\mu_b=\mu^X_b : X \rightarrow X , \quad x \mapsto b x.
\]

For a line bundle $L$ on $X$,
we set
\[
K(L)=\{ x \in X \, | \, t_x^* L \simeq L  \},
\]
where $t_x : X \rightarrow X$ is the translation by $x$ on $X$.
Equivalently,
$K(L)$ is the kernel of the group homomorphism $X \ra \widehat{X} : x \mapsto t_x^* L \otimes L^{-1}$,
where $\widehat{X}=\Pic^0 (X)$ is the dual abelian variety of $X$.
If $L$ is non-degenerate, that is, $\chi(L) \neq 0$,  
then it is known that there exist positive integers $d_1|d_2 | \cdots|d_g$ such that 
$K(L) \simeq (\bigoplus_{i=1}^g \Z/d_i\Z )^{\oplus 2}$ as abelian groups.
The vector $(d_1,\dots,d_g)$ is called the \emph{type} of $L$.

For an ample line bundle $L$ on $X$,
we call 
$(X,L)$ a \emph{polarized abelian variety}.
It is known that an ample line bundle $L$ is non-degenerate and
$\chi(L)  =  \prod_{i=1}^g d_i$ holds, where $(d_1,\dots,d_g) $ is the type of $L$.

 \vspace{2mm}
Let $D=(d_1,\dots,d_g)$ be a type,
that is, a sequence of positive integers such that $d_1|d_2|\cdots |d_g$.
We set 
\begin{align*}
\lambda (D) \coloneqq g  , \quad |D| \coloneqq \sum_{i=1}^{g} d_{i},  \quad \chi(D) \coloneqq \prod_{i=1}^g d_i.
\end{align*}
We note that $\chi(L)=\chi(D)$ holds if $(X,L)$ is a polarized abelian variety of type $D$.

 If $d_i= d_{i+1} = \dots =d_{i+m-1} = c$ for a type $D$,
we abbreviate $d_i,\dots,d_{i+m-1}$ in $D$ as $c^{(m)}$.
For example, $(1^{(2)},6) = (1,1,6) $, 
$(1^{(3)}, 2^{(4)},8) =(1,1,1,2,2,2,2,8)$
and $(1^{(0)}, 2,4,4) =(2,4,4)$.

\subsection{Basepoint-freeness thresholds}\label{subsec_BFT}

We refer the readers to  \cite{MR4157109}
for the definition of basepoint-freeness thresholds $\beta(L)=\beta(X,L)$.
In this paper,
we use the following properties to estimate $\beta(L)$. 

\begin{lem}\label{lem_divisor}
Let $(X,L)$ be a polarized abelian variety of dimension $g$.
Then 
\begin{enumerate}[(i)]
\item $0 < \beta(L) \leq 1$, and $\beta(L) <1$ if and only if $L$ is basepoint free. 
\item For an integer $n >0$, it holds that $\beta(n L ) = \beta(L) /n$.\label{item_beta(nl)}
\item If $g=1$, it holds that  $\beta(L) =1/\deg(L)$.
\item $\beta(L) \geq 1/ \sqrt[g]{\chi(L)}  $. 
\item For an abelian subvariety $Z \subset X$,
it holds that $\beta(L) \geq \beta(L|_Z)$.
Furthermore, 
\[
\beta(L|_Z) \leq \beta(L) \leq \max \left\{ \beta(L|_Z), \frac{g (L^{g-1}.Z)}{(L^g)} \right\} = \max \left\{ \beta(L|_Z), \frac{\chi(L|_Z)}{\chi(L)} \right\}  
\]
if the codimension of $Z \subset X$ is one.
\end{enumerate}
\end{lem}

\begin{proof}
(i) is proved in  \cite[Section 8]{MR4157109}.
(ii), (iii) immediately follow from definition.
(iv), (v) are proved in \cite[Lemmas 3.4, 4.3]{Ito:2020aa}.
\end{proof}

\section{Explicit description of $K(L)$ for some special $(X,L)$}\label{sec_construction}

In this section,
we construct examples of polarized abelian varieties $(X,L)$, for which we can explicitly describe $K(L)$.

\vspace{2mm}
In this section,
we fix the following setting:

\begin{sett}\label{sett_X,L}
Let $g \geq 2, k_1,\dots,k_{g-1} \geq 1$ be integers and set $k_g\coloneqq 1$.
We take 
\begin{itemize}
\setlength{\itemsep}{0mm}
\item  an elliptic curve  $E_g$, and
\item an isogeny $ f_i : E_i  \ra E_g$ from an elliptic curve $E_i$ for each $1 \leq i \leq g-1$ with $\ker f_i \simeq \Z/ k_i \Z$.
\end{itemize}

The origin of  $E_i$ is denoted by $o_i \in E_i$.
Let $X=E_1 \times E_2 \times \dots \times E_g$ and 
let $F_i = \pr_i^* (o_{i})$,
where $\pr_i : X \rightarrow E_i$ is the projection to the $i$-th factor. 
A divisor $\Gamma$ on $X$ is defined by
\[
\Gamma = \left\{ \Big(  p_1,\dots,p_{g-1}, \sum_{i=1}^{g-1} f_i(p_i) \Big) \in X \, \Big| \, p_i \in E_i \text{ for }  1 \leq i \leq g-1 \right\}.
\]
We identify $\widehat{X} =\widehat{E}_1 \times \dots \times \widehat{E}_g $ with $X$ by the natural isomorphism $\widehat{E}_i \simeq E_i $.
\end{sett}

Under \autoref{sett_X,L}, 
$\hat{f}_i : E_g \ra E_i$ denotes the dual isogeny of $f_i$.
It holds that
$
\hat{f}_i \circ f_i = \mu^{E_i}_{k_i} $ and $  f_i \circ \hat{f}_i =\mu^{E_g}_{k_i} 
$ for each $i$  (see e.g.\ \cite[Chapter III, Theorem 6.2]{MR2514094}).

\begin{lem}[{\cite[Section 5]{Ito:2020ab}}]\label{lem_gamma}
Under \autoref{sett_X,L}, the following hold.
\begin{enumerate}[(i)]
\item For $x=(p_1,\dots,p_g) \in X$,
the numerically trivial line bundle $t_x^* \calo_X(F_i) \otimes \calo_X(-F_i) \in \widehat{X}$ corresponds to 
$(o_1,\dots,o_{i-1},-p_i,o_{i+1},\dots,o_g) \in X$
under the identification $\widehat{X}=X$.
On the other hand,
$t_x^* \calo_X(\Gamma) \otimes \calo_X(-\Gamma) \in \widehat{X}$ corresponds to
$
( \hat{f}_1(A), \dots, \hat{f}_{g-1}(A),-A) \in X,
$
where $A=p_g -\sum_{i=1}^{g-1} f_i(p_i) \in E_g$.\label{item_corresponding_point}
\item It holds that $F_i^2=\Gamma^2=0$ as cycles on $X$ and
$
(F_1\cdots F_g) =1$, $(F_1 \cdots F_{i-1}.F_{i+1} \cdots F_g.\Gamma) =k_i $ for $ 1 \leq i \leq g$. 
\label{item_intersection_number}
\end{enumerate}

\end{lem}

The following lemma is proved in \cite[Lemma 4.1]{Ito:2020ab} under the assumption $a,b \geq 0$ with $(a,b) \neq (0,0)$.
The proof is essentially the same.

\begin{lem}\label{lem_example_surface}
In \autoref{sett_X,L}, 
assume $g=2$ and $k_1=k \geq 1$.
Let $ a,b $ be integers 
and let $L=L_{a,b} =aF_1 +bF_2+\Gamma$ on $X=E_ 1\times E_2$.
Then 
\begin{enumerate}[(i)]
\item  $K(L) = \{ (p,q) \in E_1 \times E_2 \, | \, \hat{f}_1(q) =(a+k)p , f_1(p)=(1+b)q \}$. \label{item_K(L)}
\item If $a +ab +bk \neq 0$, then $L$ is non-degenerate of type $(1, |a +ab +bk|)$, and hence
$K(L) \simeq (\Z/|a +ab +bk|\Z)^{\oplus 2} = (\Z/(a +ab +bk)\Z)^{\oplus 2} $.\label{item_a+ab+bk} \label{item_type}
\end{enumerate}
\end{lem}

\begin{proof}
(i) In this case, $A=q- f_1(p) $ and hence $\hat{f}_1(A) =\hat{f}_1 (q- f_1(p)) = \hat{f}_1(q) -k p $.
 By \autoref{lem_gamma} (i),
$(p,q)$ is contained in $K(L)$ if and only if 
\[
(o_1, o_2) =  a (-p, o_2) +b (o_1, -q ) + (\hat{f}_1(A) , -A) = (\hat{f}_1(q) -(a+k)p  ,   f_1(p) -(b+1)q ).
\] 
Hence we obtain (i).

\vspace{1mm}
\ni
(ii) By \autoref{lem_gamma} (ii), it holds that
$(F_i^2)=(\Gamma^2)=0, (F_1.F_2) =(F_1.\Gamma) =1, (F_2.\Gamma)=k$.
Hence  we have $\chi(L)=(L^2)/2 =a+ab+bk =d \neq 0$.

Since $\ker f \simeq \Z/k$, $f \in \Hom(E_1,E_2)$ is primitive,
that is, $f$ is not written as $c \lambda$ for some integer $c \geq 2$ and some $\lambda \in \Hom(E_1,E_2)$.
Hence $L$ is primitive as well by 
\cite[Theorem 11.5.1]{MR2062673}.
Thus the type of $L$ is $(1, |\chi(L)|) =(1,|d|)$. 
\end{proof}

The following proposition generalizes \cite[Proposition 5.1]{Ito:2020ab},
which is the case $a_1=\dots = a_{g-2}=c=1$ and $a,b \geq 0$ with $(a,b) \neq (0,0)$:

\begin{prop}\label{prop_K(L)}
Let  $L =a_1 F_1 + \dots + a_g F_g + c \Gamma$ be a line bundle for (possibly negative) integers $a_1.\dots,a_{g},c$ 
such that  
\begin{itemize}
\item $a_1,\dots,a_{g-2}, c$ are non-zero, and 
\item $a_{g-1}/c,  a_g/c$, and  $ c/a_i $ for $ 1 \leq i \leq g-2$ are integers.
\end{itemize}
Set 
$
a = a_{g-1}/c, \ b = a_g/c, \ N= 1+  \sum_{i=1}^{g-2} \frac{ c}{a_i} k_i  ,\  d= abN +bk_{g-1} +a.
$
Then 
\begin{enumerate}
\item $\chi(L) = a_1 \cdots a_{g-2} \cdot c^2 d$.
\item If $ d \neq 0$, it holds that  $K(L) \simeq \left(\bigoplus_{i=1}^{g-2} \Z/a_{i} \Z \oplus \Z/ c \Z \oplus \Z / c d \Z \right)^{\oplus 2}$.
\end{enumerate}
\end{prop}

\begin{proof}
We note that $a,b,N$ are integers by assumption.

\vspace{1mm}
\noindent
(1) By \autoref{lem_gamma} (ii), we have
\begin{align*}
\chi(L) = \frac{(L^g)}{g!} &= a_1 \cdots a_g + a_1 \cdots a_g  \sum_{i=1}^g \frac{c}{a_i} k_i\\
&= a_1 \cdots a_{g}  \left(  1+  \sum_{i=1}^g \frac{c}{a_i} k_i \right) \\
&= a_1 \cdots a_{g-2} \cdot ac \cdot bc \left(  1+  \sum_{i=1}^{g-2} \frac{c}{a_i} k_i  + \frac{c}{ac} k_{g-1}+ \frac{c}{bc} k_g\right)\\
&= a_1 \cdots a_{g-2} \cdot c^2 \left( ab N  +b k_{g-1}+ a\right) ,
\end{align*}
where we use $k_g=1$ in the last equality.

\vspace{1mm}
\noindent
(2) 
For $x=(p_1,\dots,p_g) $, we set $A = p_g -\sum_{i=1}^{g-1} f_i(p_i)$.
By \autoref{lem_gamma} (i),
\begin{align*}
x \in K(L) \quad 
\lra \quad & \sum_{i=1}^g a_i (o_1,\dots,o_{i-1},-p_i,o_{i+1},\dots,o_g) + c ( \hat{f}_1(A), \dots, \hat{f}_{g-1}(A),-A)  = o_X \\
\lra \quad &-a_i p_i + c\hat{f}_i(A) =o_i  \ \text{ for } 1 \leq i \leq g-1, \quad  -a_g p_g -cA =o_g \\
\lra \quad &-a_i p_i + \hat{f}_i( cA) =o_i  \ \text{ for } 1 \leq i \leq g-1, \quad  -a_g p_g -cA =o_g \\
\lra \quad &-a_i p_i + \hat{f}_i(-a_g p_g) =o_i  \ \text{ for } 1 \leq i \leq g-1, \quad  -a_g p_g -cA =o_g \\
\lra \quad &-a_i p_i -a_g  \hat{f}_i(p_g) =o_i  \ \text{ for } 1 \leq i \leq g-1, \quad  -a_g p_g -cA =o_g.
\end{align*}
Under the conditions $a_i p_i + a_g  \hat{f}_i(p_g) =o_i $ for $1 \leq i \leq g-2$,
it holds that
\begin{align*}
a_g p_g +cA  &=a_g p_g  + c \left(p_g - \sum_{i=1}^{g-1} f_i(p_i) \right) \\
&=a_g p_g  +c p_g -  \sum_{i=1}^{g-2} \frac{c}{a_i} f_i( a_ip_i) -  c f_{g-1}(p_{g-1})\\
&=a_g p_g  +c p_g -  \sum_{i=1}^{g-2} \frac{c}{a_i} f_i( -  a_g  \hat{f}_i(p_g) ) -  c f_{g-1}(p_{g-1})\\
&=a_g p_g  +c p_g + a_g \sum_{i=1}^{g-2} \frac{c }{a_i} k_i p_g  -  c f_{g-1}(p_{g-1})\\
&=\left( a_g +c + a_g  \sum_{i=1}^{g-2} \frac{ c}{a_i} k_i   \right) p_g -  c f_{g-1}(p_{g-1})\\
&=\left(c+ a_g N \right) p_g -c f_{g-1}(p_{g-1}),
\end{align*}
where we use $c/a_i \in \Z$ for $1 \leq i \leq g-2$ in the second equality, 
$f_i \circ \hat{f}_i = \mu_{k_i}^{E_g}$ in the fourth one, and $N= 1+  \sum_{i=1}^{g-2} \frac{ c}{a_i} k_i  $ in  the last one.
Hence $K  (L) $ coincides with
\begin{align*}
\{(p_1,\dots,p_g) \, | \,  a_{i} p_i = -a_g  \hat{f}_i(p_g)  \text{ for } 1 \leq i \leq g-1,  \ c f_{g-1}(p_{g-1}) =(c+a_gN) p_g\}.
\end{align*}
Thus we have an exact sequence of abelian groups
\begin{align}\label{eq_exact_seq_K}
0 \ra K_1 \xrightarrow{\iota} K  (L) \xrightarrow{\pi} K_2,
\end{align}
where 
\begin{align*}
K_1 &=\{(p_1,\dots,p_{g-2}) \in E_1 \times \dots \times E_{g-2} \, | \, a_{i} p_i = o_{i} \  \text{for } 1 \leq i \leq g-2\},\\
K_2&=\{ (p_{g-1},p_g) \in E_{g-1} \times E_g \, | \,  a_{g-1} p_{g-1}= -a_g   \hat{f}_{g-1}(p_g),\ c f_{g-1}(p_{g-1} ) = (c+a_gN) p_g \}\\
&=\{ (p_{g-1},p_g) \in E_{g-1} \times E_g \, | \,  ac p_{g-1}= - bc   \hat{f}_{g-1}(p_g),\ c f_{g-1}(p_{g-1} ) = (c+bcN) p_g \}
\end{align*}
and $\iota, \pi $ are the natural inclusion and projection respectively.
Furthermore,  $\pi$ is surjective since
\[
\pi' : K_2 \ra K  (L)  \quad : \quad (p_{g-1},p_g) \mapsto ( p'_1,\dots,p'_{g-2},p_{g-1}, p_g )
\]
defined by $ p'_{i} =- \frac{a_g}{a_i} \hat{f}_i(p_g)$ satisfies $\pi \circ \pi'=\id_{K_2}$.
Here, we note that $ \frac{a_g}{a_i} =  \frac{a_g}{c} \cdot \frac{c}{a_i}$ is an integer for $1 \leq i \leq g-2$.
Hence we have $ K  (L )  \simeq K_1 \oplus K_2 $.

Recall that what we need to show is
$K(L) \simeq ( \bigoplus_{i=1}^{g-2} \Z/a_i \Z \oplus  \Z/c\Z \oplus \Z/ c d\Z )^{\oplus 2}$.
For an abelian variety $A$ and an integer $m $,
we write $ A_m = \{ a \in A \, | \, ma=o_A\}$.
Then we have
$
K_1=  \bigoplus_{i=1}^{g-2} (E_i)_{a_{i} } \simeq ( \bigoplus_{i=1}^{g-2} \Z/a_i \Z )^{\oplus 2}
$
by the definition of $K_1$.
Since $ K(L) \simeq K_1 \oplus K_2 $,
the rest is to show that 
$K_2 \simeq ( \Z/ c  \Z  \oplus \Z/ cd \Z)^{\oplus 2}$.

By the definition of $K_2$, it holds that $K_2 = \mu_c^{-1} (K)$ for
\begin{align*}
K=\{ (p_{g-1},p_g) \in E_{g-1} \times E_g \, | \,  a p_{g-1}= - b   \hat{f}_{g-1}(p_g),\  f_{g-1}(p_{g-1} ) = (1+bN) p_g \}.
\end{align*}

\begin{claim}\label{claim_K_2_K}
$K_2 \simeq ( \Z/ c  \Z  \oplus \Z/ cd \Z)^{\oplus 2}$ holds if $K \simeq  (\Z/ d \Z)^{\oplus 2}$.
\end{claim}

\begin{proof}[Proof of \autoref{claim_K_2_K}]
Let $\Lambda \subset \C^2$ be a lattice such that $E_{g-1} \times E_g \simeq \C^2/\Lambda$ and $ \lambda: \C^2 \ra E_{g-1} \times E_g $ be the induced quotient map.
Then $\Lambda':= \lambda^{-1}( K) \supset \Lambda$ is a lattice such that $ \Lambda'/\Lambda \simeq K$.
If $K \simeq  (\Z/ d \Z)^{\oplus 2}$,  there exists a basis $e_1,\dots,e_4$ of $\Lambda$ such that $ \Lambda' = \Z e_1 + \Z e_2 + \Z \frac1d e_3 +  \Z \frac1d e_4$
by the theory of elementary divisors.
Since $K_2 = \mu_c^{-1} (K)$, we have $\lambda^{-1} (K_2) = \frac1c \Lambda' =  \Z \frac1c e_1 + \Z \frac1ce_2 + \Z \frac1{cd} e_3 +  \Z \frac1{cd} e_4$,
and hence $K_2 \simeq \lambda^{-1} (K_2) / \Lambda \simeq  ( \Z/ c  \Z  \oplus \Z/ cd \Z)^{\oplus 2}$.
\end{proof}

By \autoref{claim_K_2_K},
it suffices to show  that $K$   
is  isomorphic to $ (\Z/ d \Z)^{\oplus 2}$.
If $N=0$, we have
\begin{align*}
K&=\{ (p_{g-1},p_g) \in E_{g-1} \times E_g \, | \,  a p_{g-1}= - b   \hat{f}_{g-1}(p_g),\  f_{g-1}(p_{g-1} ) =  p_g \}\\
&=\{ (p_{g-1},p_g) \in E_{g-1} \times E_g \, | \,  (a +b k_{g-1}) p_{g-1}= o_{g-1},\  f_{g-1}(p_{g-1} ) =  p_g \}\\
&=\{ (p_{g-1},p_g) \in E_{g-1} \times E_g \, | \,  d p_{g-1}= o_{g-1},\  f_{g-1}(p_{g-1} ) =  p_g \}\\
&\simeq (E_{g-1})_{d} \simeq (\Z/ d \Z)^{\oplus 2},
\end{align*}
where the third equality follows from $d= abN + bk_{g-1} +a$ and $N=0$.

Hence we may assume $N \neq 0$.
%In this case, $K \simeq  (\Z/ d \Z)^{\oplus 2}$ can be shown by the argument in the proof of \cite[Proposition 5.1.(1)]{Ito:2020ab} as follows:
Consider a group
\begin{align*}
K'= \left\{ (p_{g-1},p_g ) \in E_{g-1} \times E_g \, \Big| \, a Np_{g-1}= -b N\hat{f}_{g-1}(p_g) , \ f_{g-1}(p_{g-1}) =(1+bN)p_g   \right\}.
\end{align*}
By definition, we have
\begin{align}\label{eq_NK'}
NK' :=\{ (Np_{g-1},Np_g) \, | \, (p_{g-1},p_g) \in K' \} \subset K \subset K'.
\end{align}
Under the condition $  f_{g-1}(p_{g-1}) =(1+bN)p_g $,
we have the following equivalences:
\begin{align*}
a Np_{g-1}= -b N\hat{f}_{g-1}(p_g) \ &\Leftrightarrow \ a Np_{g-1} +\hat{f}_{g-1}(f_{g-1}(p_{g-1}))= -b N\hat{f}_{g-1}(p_g) + \hat{f}_{g-1}((1+bN)p_g  )\\
&\Leftrightarrow \  (a N+k_{g-1}) p_{g-1}=  \hat{f}_{g-1}(p_g ).
\end{align*}
Hence it holds that
\[
K'= \left\{ (p_{g-1},p_g ) \in E_{g-1} \times E_g \, \Big| \,  \hat{f}_{g-1}(p_g) =(aN+k_{g-1}) p_{g-1} , \  f_{g-1}(p_{g-1}) =(1+bN)p_g  \right\}.
\]
By \autoref{lem_example_surface} (i), 
we have 
$K'  = K(L')$ for $L':=L_{aN,bN}$ 
with $k=k_{g-1}$ in \autoref{lem_example_surface}. 
By \autoref{lem_example_surface} (ii), $L'$ is of type $(1, |aN+aNbN+ bNk_{g-1}|) =(1,|Nd|)$ since $d= abN + bk_{g-1} +a$.
Hence it holds that
$K'=K(L') \simeq (\Z / Nd \Z)^{ \oplus 2}$ and $NK' \simeq (N\Z / Nd \Z)^{ \oplus 2} \simeq (\Z / d \Z)^{ \oplus 2}$.

By $|K(L)| =\chi(L)^2 =(a_1 \cdots a_{g-2} \cdot c^2d)^2$ and $|K_1|=| ( \bigoplus_{i=1}^{g-2} \Z/a_i \Z )^{\oplus 2}| =(a_1 \cdots a_{g-2})^2 $,
we have $|K_2|= |K(L)|/|K_1| = c^4 d^2$.
Since $K_2=\mu_c^{-1}(K)$, it holds that $|K|=|K_2|/|\ker \mu_c| = |K_2|/c^4=d^2$.
Since $K$ contains $NK' $ by \ref{eq_NK'} and  $|NK' | =| (\Z / d \Z)^{ \oplus 2}|=d^2$, we have $K=NK' \simeq (\Z / d \Z)^{ \oplus 2}$ and (2) follows.
\end{proof}

\section{Upper bounds of $\beta(D)$}\label{sec_any_type}

\begin{defn}\label{defn_beta(d_1,...,d_g)}
Let $D=(d_1,\dots,d_g)$ be a type,
that is, a sequence of positive integers such that $d_1|d_2|\cdots |d_g$.
We define
\[
\beta(D)=\beta(d_1,\dots,d_g) :=\inf_{(X,L)} \beta(X,L),
\]
where we take the infimum for all polarized abelian varieties $(X,L)$ of type $D=(d_1,\dots,d_g)$.
\end{defn}

\begin{rem}\label{rem_beta(D)}
For a real number $t >0$, 
$\beta(X,L) < t$ is an open condition in the moduli space of polarized abelian varieties of a fixed type by \cite[Theorem 3.1]{Ito:2020ab}.
In particular,
\begin{enumerate}[(i)]
\setlength{\itemsep}{0mm}
\item $\beta(D) < t$ holds if and only if $\beta(X,L) < t$ holds for general $(X,L)$ of type $D$. 
\item $\beta(D) =\beta(X,L)$ holds for very general $(X,L)$ of type $D$.
\end{enumerate}
\end{rem}

In  \cite{Ito:2020ab}, the author investigates $\beta(D) $ for $D=(1,\dots, 1 , d) =(1^{(g-1)} ,d)$
and gives upper bounds of $\beta(1^{(g-1)},d )$ as follows:

\begin{thm}[{\cite[Theorem 3.1]{Ito:2020ab}}]\label{thm_Ito:2020ab}
Let $d ,g \geq 1$ be integers and set $m:=\lfloor  \sqrt[g]{d} \rfloor$.
Then 
\begin{enumerate}
\item $ \beta(1^{(g-1)}, d) \leq 1/m$.
\item $ \beta(1^{(g-1)}, d)   < 1/m$ if $d \geq m^{g} + \cdots + m+1 = (m^{g+1}-1)/(m-1)$.
\end{enumerate}
\end{thm}

To generalize \autoref{thm_Ito:2020ab} to arbitrary types,
we first show the following proposition:

\begin{prop}\label{lem_flag}
Let $g \geq 2, k_1,\dots,k_{g-1} \geq 1$ be integers and set $k_g\coloneqq 1$.
Let $E_i, f_i, \hat{f}_i, X=E_1 \times \dots \times E_g, F_i, \Gamma$ be as in \autoref{sett_X,L}.
Let $L$ be a line bundle on $X$.
\begin{enumerate}
\item Let $\{ n_1,\dots,n_g\}=\{1,\dots,g\}$ and 
let $X=X_0 \supset X_1 \supset \cdots \supset X_{g-1} \supset X_g =\{o_X\}$
be a flag of abelian subvarieties defined by $X_i=\bigcap_{j=1}^i F_{n_j}$ for $0 \leq i \leq g$.
Then $L$ is ample if and only if  $\chi( L|_{X_i}) >0$ for any $0 \leq i \leq g-1$.
\item 
Let $X_i$ be as in (1).
If $L$ is ample,  it holds that
\[
\beta(X,L) \leq \max_{0 \leq i \leq g-1} \frac{\chi(L|_{X_{i+1}})}{\chi(L|_{X_i})} .
\]
\item Set $X_i = \bigcap_{j=g-i+1}^g F_j$ for $ 0\leq  i \leq g$
and let $L=a_1 F_1 + \dots  + a_g F + c \Gamma$ for integers $a_1,\dots,a_g,c$.
Then it holds that $\chi(L|_{X_i}) = a_1 \cdots a_{g-i} N_i $ for $ 0 \leq i \leq g$,
where we set $N_i =1+ \sum_{j=1}^{g-i} \frac{c}{a_j} k_j  $.

In particular, if $a_1 \cdots a_{g-i} N_i  >0$ for any $0 \leq i \leq g-1$, then $L$ is ample and 
\[
\beta(X,L) \leq \max_{0 \leq i \leq g-1}  \frac{N_{i+1}}{a_{g-i} N_i}.
\]
\end{enumerate}
\end{prop}

\begin{proof}
(1) Only if part is clear. We show the if part.

By assumption, $(F_{n_1} \cdots F_{n_i} . L^{g-i}) =(L|_{X_i}^{g-i})  =(g-i)! \cdot \chi( L|_{X_i}) $ is positive for $ 0 \leq i \leq g-1$.
To show the ampleness of $L$,
it suffices to find  an ample line bundle $H$ on $X$ such that 
$(H^i.L^{g-i}) >0$ for all $0 \leq i \leq g-1$ by \cite[Corollary 4.3.3]{MR2062673}.

Take $H= t^{m_1} F_{n_1} + \dots + t^{m_g} F_{n_g}$ for integers $m_1 > \dots > m_g >0$ and $t \gg 1$.
Then
\[
\frac{(H^i. L^{g-i})}{i!} =   (F_{n_1} \cdots F_{n_i}.L^{g-i}) \, t^{m_1 + \dots +m_i} + \text{lower terms } 
\]
for $ 0 \leq i \leq g-1$.
Since $ (F_{n_1} \cdots F_{n_i}.L^{g-i}) >0$ and $t \gg 1$,
we have $(H^i.L^{g-i}) >0$ for $ 0 \leq i \leq g-1$ and the ampleness of $L$ follows.

\vspace{1mm}
\noindent
(2)
Since $X_{i+1} \subset X_{i}$ is an abelian subvariety of codimension one,
we can applying \autoref{lem_divisor} (v) to $ (X_i, L|_{X_i}) $ repeatedly,  and obtain
\begin{align*}
\beta(L) \leq \max \left\{ \beta( L|_{X_{g-1}}), \frac{\chi(L|_{X_{g-1}})}{\chi(L|_{X_{g-2}}) } , \dots , \frac{\chi(L|_{X_1})}{\chi(L|_{X_0})} \right\}.
\end{align*}
Since $ X_{g-1}$ is an elliptic curve and $X_g$ is a point, it holds that $\beta( L|_{X_{g-1}}) = 1/\deg( L|_{X_{g-1}})  = 1/\chi( L|_{X_{g-1}}) =  \chi(L|_{X_g})/\chi( L|_{X_{g-1}}) $
by \autoref{lem_divisor} (iii).
Thus (2) follows.

\vspace{1mm}
\noindent
(3) 
In this case,
\begin{align*}
\chi(L|_{X_i}) = \frac{(F_{ g-i+1} \dots F_g.L^{g-i})}{(g-i)!} = a_1 \cdots a_{g-i} + a_1 \cdots a_{g-i} \sum_{j=1}^{g-i} \frac{c}{a_j} k_j = \left(\prod_{j=1}^{g-i}  a_j \right) \cdot N_i
\end{align*}
for $0 \leq i \leq g$.
The last statement follows from (1) and (2).
\end{proof}

Combining \autoref{prop_K(L)} and \autoref{lem_flag},
we can prove  the following theorem:

\begin{thm}\label{thm_beta_general_type}
Let $D=(d_1, \dots, d_g)$ be a type.
\begin{enumerate}
\item For integers $ k_1,\dots,k_{g-2} , M \geq 1$, it holds that
 \[
\beta(D) \leq \max \left\{  \max_{2 \leq i \leq g-1}  \frac{ M_{i+1}}{  d_{g-i} M_i} ,  \frac{M_2}{d_{g-1}M} ,   \frac{M}{d_g}\right\},
\]
where $M_i=1+\sum_{j=1}^{g-i} \frac{d_{g-1}}{d_{j}} k_j  $ for $2 \leq i \leq g$.
\item 
For an integer $k \geq 0$ ,
it holds that $\beta(D) \leq 1/(d_1+kd_{g-1})$ if
$d_g \geq (d_1+kd_{g-1})  \prod_{j=1}^{g-1} \left( 1+ k \, \frac{d_{g-1}}{d_{j}}  \right)$.
\end{enumerate}
\end{thm}

\begin{proof}
(1) We use the notation in \autoref{sec_construction}.
Consider $L=  \sum_{i=1}^{g-2} d_{i} F_i  +a d_{g-1} F_{g-1} + b d_{g-1} F_g + d_{g-1} \Gamma$ on $X$ for integers $a, b$,
that is, the case 
\[
a_i=d_i  \  \text{ for } \ 1 \leq i \leq g-2, \quad a_{g-1}=a d_{g-1}, \quad a_g=b d_{g-1}, \quad c= d_{g-1}
\]
in \autoref{prop_K(L)}.
We note that $k_1, \dots,k_{g-2}$ are already given in the statement of this theorem, but $k_{g-1} $ is not.
We will choose $a,b, k_{g-1} $ suitably later.
For this $L$,
$N_i =1+\sum_{j=1}^{g-i} \frac{c}{a_j} k_j  $ in \autoref{lem_flag} (3) 
is equal to $M_i$ for $2 \leq i \leq g$.
Hence we have 
\begin{align}\label{eq_chi_X_i}
\chi(L|_{X_i}) = a_1 \cdots a_{g-i} \cdot N_i = d_1 \cdots d_{g-i} \cdot M_i >0  
\end{align}
for $2 \leq i \leq g$.
On the other hand, 
\begin{align*}
N_1&=  N_2 +(c/a_{g-1}) k_{g-1}  =   M_2 + k_{g-1}/a,\\
N_0 &= N_1+ (c/a_{g}) k_{g}  =  M_2 + k_{g-1}/a + 1/b.
\end{align*}
Hence 
\begin{align*}
\chi(L|_{X_1}) &= a_1 \cdots a_{g-1}  \cdot N_1 \\
&= d_1 \cdots d_{g-2} \cdot  a d_{g-1} \cdot (M_2 + k_{g-1}/a) = d_1 \cdots d_{g-1} (aM_2 +k_{g-1})
\end{align*}
and 
\begin{align*}
\chi(L|_{X_0}) &=  a_1 \cdots a_{g} \cdot N_0 \\
&=  d_1 \cdots d_{g-2} \cdot   a d_{g-1} \cdot b d_{g-1} \cdot (M_2 + k_{g-1}/a + 1/b) \\
&= d_1 \cdots d_{g-1} \cdot d_{g-1} (abM_2 + b k_{g-1} + a).
\end{align*}

\begin{claim}\label{claim_type}
If $ aM_2 +k_{g-1} $ and $ d':=abM_2 + b k_{g-1} + a $ are positive, then $L$ is ample of type $(d_1,\dots,d_{g-1}, d_{g-1}d') $.
\end{claim}

\begin{proof}[Proof of \autoref{claim_type}]
Since $d_i, k_i$ are positive, so are $M_i$ and 
$\chi(L|_{X_i}) = d_1 \cdots d_{g-i} \cdot M_i $  for any $2 \leq i \leq g$.
Hence if $aM_2 +k_{g-1} $ and $abM_2 + b k_{g-1} + a $ are positive, then $\chi(L|_{X_i})$ are positive for all $0 \leq i \leq  g$
and hence $L$ is ample by \autoref{lem_flag} (1).

To determine the type of $L$, we apply \autoref{prop_K(L)} to this $L$.
In this case, $N$ in \autoref{prop_K(L)} is equal to 
\[
N=  1+  \sum_{j=1}^{g-2} \frac{ c}{a_j} k_j = 1+ \sum_{j=1}^{g-2} \frac{ d_{g-1}}{d_{j}} k_j  =M_2.
\]
Hence $d= a + ab N +bk_{g-1}$ in  \autoref{prop_K(L)} is equal to $d'$ and 
\begin{align*}
K(L) &\simeq \left(\bigoplus_{i=1}^{g-2} \Z/a_{i} \Z \oplus \Z/ c \Z \oplus \Z / c d \Z \right)^{\oplus 2} \\
&=  \left(\bigoplus_{i=1}^{g-2} \Z/d_{i} \Z \oplus \Z/ d_{g-1} \Z \oplus \Z / d_{g-1} d' \Z \right)^{\oplus 2} =  \left(\bigoplus_{i=1}^{g-1} \Z/d_{i} \Z \oplus \Z / d_{g-1} d' \Z \right)^{\oplus 2}  .
\end{align*}
Thus the type of $L$ is $(d_1,\dots,d_{g-1} , d_{g-1}d')$.
\end{proof}

Now we choose suitable $a,b,k_{g-1}$ and obtain the upper bound in (1).
Write $d_g/d_{g-1}=Mq+r$ for integers $q, r$ and set
\[
a=r, \quad b=q, \quad k_{g-1 } = M - a M_2  .
\]
Replacing $q,r$ with $q+s, r-Ms$ for $s \gg 1$,
we can take such $q,r$ so that $k_{g-1 } \geq 1$. 
Then 
\begin{align*}
aM_2 + k_{g-1} &=M > 0, \\
d'=abM_2 + b k_{g-1} + a &= b (aM_2 + k_{g-1}  ) + a = bM+a = d_g/d_{g-1} >0
\end{align*}
and hence 
$L$ is ample of type $(d_1,\dots,d_g)=D$ by \autoref{claim_type}.
By \ref{eq_chi_X_i} and 
\begin{align*}
 \chi(L|_{X_1}) &=d_1 \cdots d_{g-1} (aM_2 +k_{g-1}) = d_1 \cdots d_{g-1} M ,\\
 \chi(L|_{X_0}) &=\chi(L) =  d_1  \cdots d_g ,
\end{align*}
we have
\begin{align*}
\beta(D) \leq \beta(L) &\leq \max_{0 \leq i \leq g-1} \frac{\chi(L|_{X_{i+1}})}{\chi(L|_{X_i})} \\
&= \max\left\{ \max_{2 \leq i \leq g-1} \frac{\chi(L|_{X_{i+1}})}{\chi(L|_{X_i})}, \frac{\chi(L|_{X_{2}})}{\chi(L|_{X_1})},\frac{\chi(L|_{X_{1}})}{\chi(L|_{X_0})}  \right\}\\
&=\max \left\{  \max_{2 \leq i \leq g-1}  \frac{M_{i+1}}{  d_{g-i} M_i} ,  \frac{M_2}{d_{g-1}M} ,   \frac{M}{d_g}\right\}.
\end{align*}
(2) 
If $k=0$, we have $\beta(d_1,\dots,d_g) = d_1^{-1} \beta(d_1/d_1, d_2/d_1,\dots, d_g/d_1) \leq d_1^{-1} =(d_1+k d_{g-1})^{-1}$,
where the first equality follows from \autoref{lem_divisor} \ref{item_beta(nl)}.
Hence we may assume $k \geq 1$.

Set $k_1=k$ and $k_{i} = \left( 1+ \frac{d_{g-1}}{d_{i-1}} k \right) k_{i-1}$ for $2 \leq i \leq g-2$ inductively.
Equivalently, we set
\[
k_i= k \prod_{j=1}^{i-1} \left(1+ \frac{d_{g-1}}{d_{j}} k \right)
\]
for $1 \leq i \leq g-2$.

\begin{claim}\label{claim_1+M_i}
For these $k_1,\dots,k_{g-2}$,
$M_i=1+\sum_{j=1}^{g-i} \frac{d_{g-1}}{d_{j}} k_j  $ is equal to $\prod_{j=1}^{g-i} \left(1+ \frac{d_{g-1}}{d_{j}} k \right)$ for $2 \leq i \leq g$.
\end{claim}

\begin{proof}
We prove this claim by the descending induction on $i$.
For $i=g$, this follows from $\prod_{j=1}^{g-i} \left(1+ \frac{d_{g-1}}{d_{j}} k \right) =1$ and $M_g=1$.
If $M_{i+1}=\prod_{j=1}^{g-i-1} \left(1+ \frac{d_{g-1}}{d_{j}} k \right)$ holds,
then 
\begin{align*}
M_{i} &= M_{i+1}+ \frac{d_{g-1}}{d_{g-i}} k_{g-i}\\
&=\prod_{j=1}^{g-i-1} \left(1+ \frac{d_{g-1}}{d_{j}} k \right) + \frac{d_{g-1}}{d_{g-i}}  \cdot k \prod_{j=1}^{g-i-1} \left(1+ \frac{d_{g-1}}{d_{j}} k \right)\\
&=\prod_{j=1}^{g-i} \left(  1 + \frac{d_{g-1}}{d_{j}} k \right)
\end{align*}
and this claim follows.
\end{proof}

By \autoref{claim_1+M_i},
we have
\[
\frac{ M_{i+1}}{  d_{g-i}  M_i} = \frac{1}{d_{g-i}  \left(  1 + \frac{d_{g-1}}{d_{g-i}} k \right)} = \frac{1}{d_{g-i} + d_{g-1} k} \leq \frac{1}{d_{1} + d_{g-1} k}
\]
for $2 \leq i \leq g-1$.
Take 
\[
M=(1 + k) M_2 = (1 + k)  \prod_{j=1}^{g-2} \left(  1 + \frac{d_{g-1}}{d_{j}} k \right) = \prod_{j=1}^{g-1} \left(  1 + \frac{d_{g-1}}{d_{j}} k \right) .
\]
Then
\[
\frac{M_2}{d_{g-1}M}  = \frac{1}{d_{g-1} (1+k)} \leq \frac{1}{d_{1} + d_{g-1} k}
\]
and hence

\begin{align*}
\beta(D) &\leq \max \left\{  \max_{2 \leq i \leq g-1}  \frac{M_{i+1}}{  d_{g-i} M_i} ,  \frac{M_2}{d_{g-1}M} ,   \frac{M}{d_g}\right\} 
\leq  \max \left\{   \frac{1}{d_{1} + d_{g-1} k} ,   \frac{\prod_{j=1}^{g-1} \left(  1 + \frac{d_{g-1}}{d_{j}} k \right) }{d_g}\right\}.
\end{align*}
Thus we have $\beta(D) \leq 1/(d_1+kd_{g-1})$ if
$d_g \geq (d_1+kd_{g-1})  \prod_{j=1}^{g-1} \left( 1+ k \, \frac{d_{g-1}}{d_{j}}  \right)$.
\end{proof}

\begin{rem}\label{ex_g=3_gen_type_variant}
In the proof of \autoref{thm_beta_general_type}, we consider line bundles of the form $\sum_{i=1}^{g-1} d_i F_i + a d_{g-1} F_{g-1} +b d_{g-1} F_{g} + d_{g-1} \Gamma$. 
Considering other line bundles, we could obtain other bounds as follows:

Consider the case $D=(d_1,d_2,d_3)=(1,c,cd)$. 
For integers $ k_1 , M\geq 1 $ with $c k_1 -1>0$,
it holds that
\begin{align}\label{eq_variant}
\beta(1,c,cd) =\max \left\{  \frac{1}{  c k_1-1} , \frac{c k_1 -1}{c M} ,  \frac{M}{cd}\right\}.
\end{align}

To show this,
we consider a line bundle $L$ of the form $L= -F_1 + ac F_2 +bc F_3 + c\Gamma$. 
We take $a,b $ which satisfy
$d= bM -a$ and $a,b \gg 1$.
We take $k_2 = a (c k_1-1) -M $, which is positive since $a $ is sufficiently large.
Then we have
\begin{align*}
\chi(L|_{F_2 \cap F_3}) &= (L.F_2.F_3) = ck_1 -1 >0,\\
\chi(L|_{ F_3}) &= (L^2.F_3)/2! =  ac(c k_1 -1) -ck_2 = cM >0,\\
\chi(L) &= (L^3)/3! = bc( a(c k_1 -1) -k_2 ) -ac^2 = bc^2M -ac^2 =c^2d >0.
\end{align*}
Hence $L$ is ample by \autoref{lem_flag} (1).
We can check that the type of $L$ is $(1,c,cd)$ by \autoref{prop_K(L)}, 
and hence \ref{eq_variant} follows from \autoref{lem_flag} (2).

For example,
we obtain $\beta(1,c,c ) < 1/2$ for $c \geq 5$
by taking $k_1=1, M=2$ in \ref{eq_variant}.
\end{rem}

We gives two corollaries of \autoref{thm_beta_general_type}.
The first one is a refinement of \autoref{thm_Ito:2020ab}:

\begin{cor}\label{cor_refinement}
Let $D=(d_1,\dots,d_g)=(1,\dots,1,d)$ for $d \geq 1$
and let $A_1,\dots,A_g $ be positive integers with $A_g=1$.
Then it holds that
\[
\beta(D) \leq \max \left\{ \frac{A_g}{A_{g-1}}, \frac{A_{g-1}}{A_{g-2}} \dots, \frac{A_2}{A_1}, \frac{A_1}{d} \right\}.
\]
\end{cor}

\begin{proof}
If $A_i \leq A_{i+1}$ for some $i$, then the upper bound in the statement is greater than or equal to one, and hence the inequality follows from $\beta(D) \leq 1$.
Thus we may assume $ A_i  > A_{i+1}$ for any $1 \leq i \leq g-1$.

We use the notation in \autoref{thm_beta_general_type}.
We take $k_i = A_{g-i}- A_{g-i+1} \geq 1$ for $1 \leq i \leq g-2$ and $M=A_1$.
Since $d_1 = \dots =d_{g-1} =1$, we have
\begin{align*}
M_i = 1+ \sum_{j=1}^{g-i} k_j =  1 +\sum_{j=1}^{g-i} (A_{g-j}- A_{g-j+1}  ) =1 + A_i -A_g =A_i
\end{align*}
for $2 \leq i \leq g$.
Hence this corollary follows from  \autoref{thm_beta_general_type} (1).
\end{proof}

If we take $A_i= m^{g-i}$ for $m = \lfloor \sqrt[g]{d} \rfloor$ in \autoref{cor_refinement},
we can recover \autoref{thm_Ito:2020ab} (1).
Similarly,
\autoref{thm_Ito:2020ab} (2) can be recovered by taking $A_i= m^{g-i} + \dots + m+1$.

\begin{rem}\label{rem_general_M_i}
By \autoref{thm_beta_general_type} and \autoref{cor_refinement},
it might be interesting to ask whether 
\[
\beta(D) \leq \max \left\{ \frac{A_g}{d_1A_{g-1}}, \frac{A_{g-1}}{d_2A_{g-2}} \dots, \frac{A_2}{d_{g-1}A_1}, \frac{A_1}{d_g} \right\}
\]
holds or not
for any type $D=(d_1,\dots,d_g)$ and any positive integers $A_1,\dots,A_g $ with $A_g=1$.

\end{rem}

The second corollary gives an asymptotic behavior of $\beta(d_1,\dots,d_g)$ for $d_{g} \gg d_{g-1}$:

\begin{cor}\label{cor_asymptotic}
Fix positive integers $d_1|d_2|\cdots|d_{g-1}$.
Then it holds that 
\begin{align*}
\lim_{
\begin{subarray}{c}
d_g \ra \infty \\
d_{g-1}|d_g 
\end{subarray}
}
\frac{\beta(d_1,\dots, d_g)}{1/\sqrt[g]{d_1 \cdots d_g}} = 1.
\end{align*}
\end{cor}

\begin{proof}
Since $\beta(d_1,\dots,d_g) \geq 1/\sqrt[g]{d_1 \cdots d_g}$ by \autoref{lem_divisor} (iv),
the inequality $\lim \frac{\beta(d_1,\dots, d_g)}{1/\sqrt[g]{d_1 \cdots d_g}} \geq 1$ follows. % from $\beta(d_1,\dots,d_g) \geq 1/\sqrt[g]{d_1 \cdots d_g}$.
 
 To show the converse inequality,
 set $k= \lfloor \sqrt[g]{d_1 \cdots d_g} /d_{g-1} \rfloor -1$ in \autoref{thm_beta_general_type} (2).
 If $d_g \gg d_{g-1}$, then $k \geq 0$ and  
 \[
 d_1 \cdots d_g  \geq \left( (1+k)d_{g-1} \right)^g \geq (d_1 + k d_{g-1}) \prod_{j=1}^{g-1} (d_{j}+k d_{g-1})
 \]
 since $d_1 \leq d_2 \leq \cdots \leq d_{g-1}$.
 Thus we have $\beta(d_1,\dots,d_g) \leq 1/(d_1+k d_{g-1})$
by \autoref{thm_beta_general_type} (2) and hence 
\[
\lim_{
\begin{subarray}{c}
d_g \ra \infty \\
d_{g-1}|d_g 
\end{subarray}
}
\,  \frac{\beta(d_1,\dots, d_g)}{1/\sqrt[g]{d_1 \cdots d_g}}
\leq 
\lim_{
\begin{subarray}{c}
d_g \ra \infty \\
d_{g-1}|d_g 
\end{subarray}
}
\, \frac{\sqrt[g]{d_1 \cdots d_g}}{d_1+k d_{g-1}} = \lim_{
\begin{subarray}{c}
d_g \ra \infty \\
d_{g-1}|d_g 
\end{subarray}
}
\,
\frac{\sqrt[g]{d_1 \cdots d_g}}{d_1+( \lfloor \sqrt[g]{d_1 \cdots d_g} /d_{g-1} \rfloor -1) d_{g-1}} =1.
\]
\end{proof}

By this corollary,
we see that  \autoref{thm_beta_general_type}  gives  asymptotically sharp bounds for $d_g \gg d_{g-1}$.

\begin{rem}\label{rem_related_results}
We recall some related known results.

\vspace{1mm}
\noindent
(1) %In \cite[Corollary B]{MR2833789},  R.~Lazarsfeld, G.~Pareschi and M.~Popa prove that  $L$ satisfies $(N_p)$  if $\chi(L) > \frac{(4g)^g}{2g!} (p+2)^g$ and $(X,L)$ is very general.
%In fact, their proof shows that $\beta(D) \leq \frac{4g}{\sqrt[g]{2g! } \sqrt[g]{d_1 \cdots d_g}}$ for arbitrary type $D=(d_1,\dots,d_g)$.
As explained in \cite[Proposition 3.1]{Ito:2020aa}, 
$\beta(X,L) \leq g/\ep(X,L)$ holds for a polarized abelian variety $(X,L)$ by \cite[Lemma 1.2]{MR2833789} and \cite[Proposition 1.4]{MR4114062},
where $\ep(X,L) $ is the Seshadri constant of $(X,L)$.
Hence we have
$\beta(D) \leq \frac{4g}{\sqrt[g]{2g! } \sqrt[g]{d_1 \cdots d_g}}$ for a type $D=(d_1,\dots,d_g)$
by the upper bound of the Seshadri constant in \cite[Theorem 1]{MR1660259}.

\vspace{1mm}
\noindent
(2)  Jiang  \cite[Theorem 2.9]{MR4395094} proves that $\beta(X,L) < t$ holds for a very general abelian variety $(X,L)$ of type $(d_1,\dots,d_g)$ 
if $t > \frac{g}{\sqrt[g]{g!} \sqrt[g]{d_1 \cdots d_g}}$ and $t d_i \leq g/(g-i) $ for $1 \leq i \leq g-1$.
We note that the condition that $(X,L)$ is very general is explicit there, i.e.\
 \cite[Theorem 2.9]{MR4395094} just
requires the space of Hodge classes to be of dimension one in each degree.

\vspace{2mm}
Though  \autoref{thm_beta_general_type} cannot recover these results,
the bounds obtained from \autoref{thm_beta_general_type} are better than those obtained from the above results
if $d_g \gg d_{g-1}$ 
since $  \frac{g}{\sqrt[g]{g!}} \sim e$ for $g \gg 1$ by Stirling's formula.
\end{rem}

In the last of this section,
we give an explicit sufficient condition for $\beta(D) < 1/2$ when $ D $ is of the form $(1^{(n-1)}, 2^{(n')},d) $.
Recall that $(1^{(n-1)}, 2^{(n')},d) $ denotes the type $(d_1,\dots,d_{n+n'})$ with $d_1= \dots =d_{n-1}=1, d_n=\dots = d_{n+n'-1} =2$ and $d_{n+n'} =d$.
We will use this  lemma in the proof of \autoref{thm_proj_normality}.

\begin{lem}\label{cor_111222d}
Let $D=(1^{(n-1)}, 2^{(n')},d)$ for $n \geq 1, n' \geq 0$.
Then $\beta(D) < 1/2 $ holds if $d \geq 2^{n+1} +2n' -1$.
\end{lem}

\begin{proof}
Set $g=n+n'$.
If $n' =0$, this lemma follows from \autoref{thm_Ito:2020ab}.
Thus we may assume $n' \geq 1$ and hence $d_{g-1} =2$.

We apply \autoref{thm_beta_general_type} (1) to
\[
k_i= 2^{i-1} \ \text{ for } \ 1 \leq i \leq n-1, \quad k_n= \dots = k_{g-2} =1, \quad M=2^{n}  + n'-1.
\]
For $2 \leq i \leq n'$, we have $n \leq g-i \leq  g-2 $ and hence $d_{g-i} =2$. Thus it holds that 
\begin{align*}
 \frac{ M_{i+1}}{  d_{g-i} M_i} < \frac1{d_{g-i}} =\frac12
\end{align*}
for $2 \leq i \leq n'$ by $ M_{i+1} < M_i$.
For $n' +1\leq i \leq g$,  
we have $g-i \leq n-1$ and hence
\begin{align*}
 M_i = 1+\sum_{j=1}^{g-i} \frac{d_{g-1}}{d_{j}} k_j =1+\sum_{j=1}^{g-i} \frac{2}{1} \cdot 2^{j-1} =2^{g-i+1}-1.
\end{align*}
Thus it holds that
\begin{align*}
 \frac{ M_{i+1}}{  d_{g-i}  M_i}  =  \frac{2^{g-i}-1}{  1 \cdot (2^{g-i+1}-1)}  < \frac12
\end{align*}
for $n'+1 \leq i \leq g-1$.
Since 
\begin{align*}
M_2 &= 1+\sum_{j=1}^{g-2} \frac{d_{g-1}}{d_{j}} k_j  \\
&= 1+\sum_{j=1}^{n-1} \frac{d_{g-1}}{d_{j}} k_j   + \sum_{j=n}^{g-2} \frac{d_{g-1}}{d_{j}} k_j  \\
&=2^{n} -1 + g-1 -n =2^{n}  + n'-2 =M-1,
\end{align*}
we have 
\begin{align*}
\frac{M_2}{d_{g-1} M} = \frac{M-1}{2 M}  <\frac12. 
\end{align*}
Hence \autoref{thm_beta_general_type} implies
$\beta(D) < 1/2$ if $M/d_g <1/2$,
that is, $d_g > 2M= 2^{n+1}  + 2n'-2$. 
\end{proof}

\section{Decompositions to products of lower dimensions}\label{sec_decomposition}

In this section,
we show an easy proposition,
which could give better bounds of $\beta(D)$ than those obtained by \autoref{thm_beta_general_type} for some $D$. 
We also give an application of the obtained bounds to the Infinitesimal Torelli Theorem  for hypersurfaces in abelian varieties.

\vspace{2mm}
Let $(X,L)$ be a general polarized abelian variety of type $(d_1,\dots,d_g)$ of dimension $g$.
In \cite[Theorem 1.1]{MR2231163}, L.~Fuentes Garc\'{\i}a shows that $L$ is basepoint free if $d_1 + \dots + d_g \geq  2g $.
The idea of the proof is considering products of polarized abelian varieties and reducing the general case 
to the type $(1,\dots, 1,d)$, which are treated in \cite{MR1299059}.
His idea can be used to bound $\beta(D)$ as follows.

\vspace{2mm}
Let 
$D_1=(d_{1},\dots,d_{g})$ and  $D_2=(d'_{ 1},\dots,d'_{g'})$ be two types. 
We define $D_1 \times D_2 = (\delta_1,\dots, \delta_{g+g'})$ by the conditions $\delta_1 | \cdots | \delta_{g+g'}$ and 
\[
\bigoplus_{i=1}^g \Z/d_{i} \Z \oplus \bigoplus_{i'=1}^{g'} \Z/d'_{i'}\Z \simeq \bigoplus_{j=1}^{g'+g''} \Z/\delta_j \Z.
\]

If $(X_1,L_1)$ and $(X_2,L_2)$ are polarized abelian varieties of type $D_1$ and $D_2$  respectively,
then $D_1 \times D_2 $ is nothing but the type of the product $(X_1 \times X_2, p_1^* L_1 \otimes p_2^* L_2) $
since $K(p_1^* L_1 \otimes p_2^* L_2) \simeq K(L_1) \oplus K(L_2)$.

\begin{prop}\label{prop_beta(d_1,...,d_g)}
For 
types $D_1$ and $D_2$,
it holds that $\beta(D_1 \times D_2) \leq \max\{ \beta (D_1) , \beta(D_2)\}$.
\end{prop}

\begin{proof}
Let $(X_1,L_1)$ and $(X_2,L_2)$ be very general polarized abelian varieties of type $D_1$ and $D_2$  respectively.
Then we have
\begin{align*}
\beta (X_1 \times X_2, p_1^* L_1 \otimes p_2^* L_2) = \max\{ \beta (X_1,L_1) , \beta(X_2,L_2)\}  = \max\{ \beta (D_1) , \beta(D_2)\},
\end{align*}
where the first equality follows from \cite[Lemma 4.3]{Ito:2020aa} and the second one follows from \autoref{rem_beta(D)} (ii).
Since $\beta (D_1 \times D_2) \leq \beta (X_1 \times X_2, p_1^* L_1 \otimes p_2^* L_2)$ by definition,
this proposition follows.
\end{proof}

Combining \autoref{prop_beta(d_1,...,d_g)} with \autoref{thm_Ito:2020ab},
we could obtain a better bound of $\beta(D)$ than the bounds by \autoref{thm_beta_general_type} for some $D$.
We illustrate the case $g=3$:

\begin{ex}\label{ex_(1,c,cd)}
(1) Take positive integers $d_2 | d_3$ and
consider the type $(1,d_2,d_3)$.
Since $(1,d_2,d_3) =(d_2) \times (1,d_3)$,
we have $\beta(1,d_2,d_3) \leq  \max\{\beta(d_2) , \beta(1,d_3)\} =\max\{1/d_2, \beta(1,d_3)\} $.

Hence if $d_2 \geq m, d_3 \geq m^2$ for an integer $m \geq 1$,
then $\beta(1,d_2,d_3) \leq 1/m$ by \autoref{thm_Ito:2020ab}.
In particular,
$\beta(1,m,m^2) =1/m$ holds since $\beta(1,m,m^2)  \geq 1/\sqrt[3]{1\cdot m \cdot m^2}$ by \autoref{lem_divisor} (iv).
It is not so difficult to see that we cannot obtain this bound from \autoref{thm_beta_general_type} nor results in \autoref{rem_related_results}.

\vspace{1mm}
\noindent
(2) (1) can be generalized as follows:
Let $a,b,c \geq 1$ be integers such that $\mathrm{gcd}(a,b) =1$,
and consider $(a) \times (b, bc)$.
By $\mathrm{gcd}(a,b) =1$, we have 
\[
\Z /a \Z \oplus (\Z/b \Z \oplus \Z/ bc \Z ) \simeq \Z / ab \Z \oplus  \Z/ bc \Z \simeq \Z / nb \oplus \Z / (abc/n)\Z,
\]
where $n =\mathrm{gcd}(a,c)$. 
Thus $(a) \times (b, bc) = (1, nb, abc/n)$ and we have 
\[
\beta (1,nb,abc/n) \leq \max\{ \beta(a), \beta(b,bc)\} =\max \{ 1/a,  \beta(1,c)/b \}.
\]

For example,
\begin{itemize}
\item (1) is nothing but the case $a=d_2,b=1,c=d_3 $.
\item If $a=3,b=2,c=3 $, then we have $\beta(1,6, 6) \leq \max \left\{ 1/3,  \beta(1,3)/2 \right\} =1/3$
since $\beta(1,3) =  2/3  $ by \cite[Table 1]{Ito:2020ab}.
This bound does not follow from \autoref{thm_beta_general_type}.
\item If $a=5, b=2, c=8$, then we have $\beta(1,2, 80) \leq \max \left\{ 1/5,  \beta(1,8)/2 \right\} \leq 1/5$
since $\beta(1,8) \leq 3/8  $ by \cite[Table 1]{Ito:2020ab}.
This bound follows from \autoref{thm_beta_general_type} as well.
\end{itemize}

\vspace{1mm}
\noindent
(3) 
On the other hand,
we cannot recover \autoref{thm_beta_general_type} from \autoref{thm_Ito:2020ab}  and \autoref{prop_beta(d_1,...,d_g)} in general.
For example, the following are all the decompositions of $(1,2,2p) $ for a prime number $p \geq 3$:
\[
(1) \times (2,2p),\quad (2) \times (1,2p),  \quad (p) \times (2,2), \quad (2p) \times (1,2).
\]
By these decompositions,
we only have $\beta(1,2,2p) \leq 1/2$,
though $\beta(1,2,2p) \sim 1/\sqrt[3]{4p}$ for $ p \gg 1$ by \autoref{cor_asymptotic}.
\end{ex}

For a polarized abelian variety $(X,L)$,  $L$ is basepoint free if and only if
$\beta(L) <1 $ by \autoref{thm_Bpf_threshold} (1).
Hence \cite[Theorem 1.1]{MR2231163} can be rephrased as:
For a type $D=(d_1,\dots, d_g)$,
$\beta(D) < 1$ holds if $d_1 + \dots + d_g \geq  2g $.
Using  \autoref{prop_beta(d_1,...,d_g)},
we can show a little bit more:

\begin{prop}\label{prop_refiniment}
Let $D=(d_1,\dots, d_g)$ be a type with $d_1 + \dots + d_g \geq  2g $.
Then 
\begin{enumerate}
\item $\beta(D) \leq  g/(g+1)$,
\item $\beta(D) \leq (g-1)/g$ if $g \geq 2$ except for  $(d_1,\dots,d_g) = (1,\dots,1,g+1)$,
\item $\beta(D) <  (g-1)/g$ if $g \geq 2$ except for
\begin{itemize}
\item $g=2$ and $ (d_1,d_2) = (1,3), (1,4), (1,5) , (1,6), (2,2), (2,4)$, or
\item $g=3 $ and $ (d_1,d_2,d_3) = (1,1,4),(1,1,5), (1,1,6) , (1,3,3)$, or 
\item $g \geq 4$ and $  (d_1,\dots,d_g) = (1,\dots,1,g+1) , (1,\dots,1,g+2), (1,\dots,1,2,g)$.
\end{itemize}
\end{enumerate}
\end{prop}

\begin{proof}
As in \autoref{subsec_type}, we set
$\lambda (D) \coloneqq g  , |D| \coloneqq \sum_{i=1}^{g} d_{i}$. 
Using this notation,
the condition $d_1 + \dots + d_g \geq  2g $ is written as $|D| \geq 2  \lambda (D)$. 
We prove this proposition combining \autoref{prop_beta(d_1,...,d_g)} with the bounds
\begin{align}\label{eq_g/d}
 \beta(1,\dots,1,d) &=\beta(1^{(g-1)} , d) \leq \max\{(g-1)/g, g /d\} \quad \text{for} \quad g \geq 1,\\
\label{eq_g+1/d}
\beta(1,\dots,1,d) &=\beta(1^{(g-1)} , d) \leq \max\{(g-2)/(g-1), (g+1) /d\} \quad \text{for} \quad g \geq 3.
\end{align}
We note that \ref{eq_g/d} follows from \autoref{cor_refinement} by taking $A_i = g-i+1$.
Similarly, \ref{eq_g+1/d} follows from \autoref{cor_refinement} by taking $ A_1=g+1$ and  $A_i = g-i+1$ for $2 \leq i \leq g$.

\vspace{2mm}

 We show (1)  by the induction on $\lambda(D)$.
When $\lambda(D)=1$, (1) follows from \autoref{lem_divisor} (iii). 

Assume $g \geq 2$ and (1) holds for $\lambda(D)  \leq g-1$.
Let $D$ be a type such that $\lambda(D)=g$ and $ |D| \geq 2 \lambda(D)=2g$.
Set  $n=\min\{ 1 \leq i \leq g \, | \, d_i \geq 2\}$ and hence we have $D=(1^{(n-1)} , d_n ,\dots,d_g) $ with $d_n \geq 2$.
We note that $d_g \geq 2$ by $ |D| \geq 2g$ and hence $\{ 1 \leq i \leq g \, | \, d_i \geq 2\}$ is not empty.

\vspace{1mm}
\noindent
{\bf Case 1}: $d_g \geq n+1$.
In this case,
we consider a decomposition 
\[
D =D_1 \times D_2 = (1^{(n-1)}, d_g)  \times  (d_{n},\dots, d_{g-1}).
\]
By \ref{eq_g/d}, we have $\beta(D_1) \leq \max\{(n-1)/n, n/d_g\} \leq n/(n+1)$.
We also have $\beta(D_2) \leq 1/d_{n} \leq 1/2$.
Hence 
\begin{align}\label{eq_case1}
\beta(D) 
\leq  \max \left\{ (n-1)/n, n/d_g, 1/d_n \right\} 
\leq  \max \left\{ n/(n+1), 1/2 \right\} \leq g/(g+1).
\end{align}

\vspace{1mm}
\noindent
{\bf Case 2}: $d_g \leq n$.
We consider 
\[
D =D_1 \times D_2 =(1^{(d_g-2)},d_g) \times  (1^{(n-d_g+1)}, d_{n},\dots, d_{g-1}) .
\]
Since $2 \leq d_g \leq n\leq g$, 
we have  $1 \leq \lambda(D_i) \leq g-1$ for $i=1,2$.
Furthermore,
it holds that  $|D_1| = 2 d_g -2 =2  \lambda (D_1)  $ and $| D_2| \geq 2  \lambda (D_2) $,
 where the latter follows from the former and $ |D| \geq 2 \lambda(D)$.
Hence we have 
\begin{align*}
 \beta(D_i) \leq \lambda(D_i) / (\lambda (D_i) +1)  \leq (g-1)/g  \quad \text{for} \quad i=1,2
\end{align*}
by induction hypothesis.
Thus it holds that  $\beta(D) \leq \max \left\{ \beta(D_1), \beta(D_2) \right\}  \leq (g-1)/g < g/(g+1)$.

In both cases, we have $\beta(D) \leq g/(g+1)$ and hence (1) follows.

\vspace{2mm}
\noindent
(2) Assume $g \geq 2$ and we use the above notation.
In Case 2, we always have $\beta(D)\leq (g-1)/g $.
In Case 1, it holds that 
\[
\beta(D) \leq  \max \left\{ (n-1)/n, n/d_g, 1/d_n \right\}  \leq \max \left\{ (g-1)/g, n/d_g\right\} 
\]
by \ref{eq_case1},  $n \leq g$ and $ d_n \geq 2$.
Hence 
$\beta(D) > (g-1)/g$ could happen only when $(g-1)/g  < n/d_g  \leq n /(n+1)$,
where the latter inequality follows from the assumption $d_g \geq n+1$ in Case 1.
Since $n \leq g$, this implies $n=g$ and $d_g=g+1$, that is, $D=(1^{(g-1)},g+1)$.
Hence (2) follows.

\vspace{2mm}
\noindent
(3) For $g=2$, (3) follows from 
\autoref{lem_divisor} (ii) and \autoref{thm_Ito:2020ab}.
Hence we may assume $g \geq 3$ and $D$ is not in the list of (3).

 If $d_g =2$, then $D=(2,\dots,2) $ since  $2g \leq \sum_i d_i \leq g d_g =2g$.
In this case, we have $\beta(D)=1/2< (g-1)/g$.

If $n=g$, that is, if $D=(1^{(g-1)},d)$,
then $\beta(D) \geq (g-1)/g$ could happen only when $(g+1)/d \geq (g-1)/g$ by \ref{eq_g+1/d}.
This condition is equivalent to
\[
d \leq \frac{g(g+1)}{g-1} =g+2 + \frac{2}{g-1},
\]
that is, $d \leq 6$ for $g=3$, and $d \leq g+2$ for $g \geq 4$.
These types are listed in (3) as exceptions.

Thus we may assume $d_g \geq 3$ and $n \leq g-1$.

In Case 2,
$\lambda(D_1) =d_g -1 \leq n-1 \leq g-2$ and $\lambda(D_2) =g-d_g +1 \leq g-2$.
Hence $\beta(D_i) \leq \lambda(D_i) /(\lambda(D_i)+1) < (g-1)/g$ for $i =1,2$ by (1) and we have $\beta(D) < (g-1)/g$.

In Case 1, it holds that
\[
\beta(D) \leq  \max \left\{ (n-1)/n, n/d_g, 1/d_n \right\}  \leq \max \left\{ (g-2)/(g-1), n/d_g\right\} 
\]
by \ref{eq_case1},  $n \leq g-1$, $g\geq 3$ and $ d_n \geq 2$.
Hence 
$\beta(D) \geq (g-1)/g$ could happen only when $ (g-1)/g  \leq n/d_g \leq n /(n+1)$,
which implies $ n=g-1, d_g=g$.
Thus $D = (1^{(g-2)},d_{g-1},g)$ with $d_{g-1} \geq 2$.
Since $D$ is not in the list of  (3), we have $g \geq 4$ and $d_{g-1} \geq 3$.
Consider 
\[
D=D_1 \times D_2 =(1,d_{g-1} ) \times (1^{(g-3)}, g).
\]
Since $d_{g-1} \geq 3$, we have $\beta(D_1 ) \leq 2/3 $ by \ref{eq_g/d}.
Since $|D_2| = 2g -3 \geq 2 \lambda(D_2)$,
we have $\beta(D_2) \leq \lambda(D_2)/(\lambda(D_2)+1) =  (g-2)/(g-1) $ by (1).
Hence  it holds that $\beta(D) \leq \max \{ 2/3, (g-2)/(g-1)\} < (g-1)/g $ by $g \geq 4$.

Thus  we have $\beta(D) < (g-1)/g $ in both Case 1 and Case 2 and (3) follows.
\end{proof}

For a polarized abelian variety $(X,L)$ of dimension $g$, 
Blo{\ss} \cite[Theorem 1.1]{Bloss:2019aa} shows that the Infinitesimal Torelli Theorem holds for any smooth $Y \in |L|$ if $X$ is simple and $h^0(L) > (g/(g-1))^{g} \cdot g!$.
We say that the Infinitesimal Torelli Theorem holds for $Y$ if a suitable period map of a Kuranishi family of $Y$ is an immersion, 
which is equivalent to the injectivity of the natural map $H^1(Y , T_Y) \rightarrow \Hom(H^0(Y,K_Y), H^1(V, \Omega^{\dim Y-1}_Y) )$ (see \cite{Bloss:2019aa} for detail).
As a corollary of \autoref{prop_refiniment}, we can show \autoref{cor_ITT}.

\begin{proof}[Proof of \autoref{cor_ITT}]
The case $g=2$ can be shown as in \cite[Section 3]{Bloss:2019aa}:
It is known that the Infinitesimal Torelli Theorem does not hold for a smooth curve $C$
if and only if $C$ is an hyperelliptic curve with genus greater than two.
Hence if there exists a smooth $Y \in |L|$ for which the Infinitesimal Torelli Theorem does not holds,
the type of $(X,L)$ must be one of $ (1,2), (1,3), (1,4)$ by \cite[Theorem 2.8]{MR3968899}.
Since we assume $d_1+d_2 \geq 2g=4$ and $(d_1,d_2) \neq (1,3) ,(1,4)$, this theorem holds in the case $g=2$.

Hence we may assume $g \geq 3$.
By \cite[Lemma 2.1]{Bloss:2019aa}, it suffices to show the surjectivity of $H^0(L) \otimes H^0(L^{g-1}) \rightarrow H^0(L^{g}) $.
A similar argument as in the proof of  \cite[Corollary 8.2 (b)]{{MR4157109}}  implies that the natural map $H^0(L) \otimes H^0(L^{g-1}) \rightarrow H^0(L^{g}) $ is surjective if $ \beta(L) < (g-1)/g$ (see also \cite[Theorem 1.2 (2)]{Ito:2021aa}).
Hence this theorem follows from \autoref{rem_beta(D)} (i) and \autoref{prop_refiniment} (3).
\end{proof}

By the inequality $(d_1+ \dots +d_g)/g \geq \sqrt[g]{d_1\dots d_g} $ between the arithmetic mean and the geometric mean,
the condition $d_1 +\cdots + d_g \geq 2 g$ in \autoref{cor_ITT}  is satisfied if $d_1\cdots d_g \geq 2^{g} $.
For $g \geq 2$,
we also have $(g/(g-1))^g \cdot g! \geq 2^g$,
which can be checked directly for $g=2,3,4$ and follows from $g! \geq \sqrt{2 \pi g} \,(g/e)^g \geq (g/e)^g$ for $g \geq 5$.
Hence the condition $h^0(L) =d_1 \cdots d_g> (g/(g-1))^{g} \cdot g!$ in \cite{Bloss:2019aa} implies $d_1 + \dots +d_g \geq 2g $.
Since the types in the list of \autoref{cor_ITT} do not satisfy  $d_1 \cdots d_g > (g/(g-1))^{g} \cdot g!$,
\autoref{cor_ITT} improves the bound in \cite{Bloss:2019aa} for general $(X,L)$.
However,
we do not know whether the statement of \autoref{cor_ITT} holds or not for simple abelian varieties $X$ in general.

\section{Types with $\beta(D) <1/2 $}\label{sec_projective_normal}

In this section, 
we use the notation
$
\lambda (D) =g  , |D| =\sum_{i=1}^{g} d_{i},   \chi(D) =  \prod_{i=1}^g d_i
$
in \autoref{subsec_type} 
for a type $D=(d_1,\dots, d_g)$.

The purpose of this section is to prove a slight improvement
%\footnote{The converse of \autoref{thm_Bpf_threshold} (2) does not hold in general \cite[Example 2.4]{Ito:2021aa}.} 
of  \autoref{thm_proj_normality} 
as follows:

\begin{thm}\label{thm_beta<1/2}
Let $D=(d_1,\dots,d_g) $ be a type.
Then  $\beta(D) < 1/2$ holds if $\chi(D) \geq 2^{2g-1}$ and $D \neq (2,4^{(g-1)})$.

In particular,  $\beta(D) < 1/2$ holds  if $\chi(D) > 2^{2g-1}$.
\end{thm}

\begin{rem}\label{rem_sharp_beta}
Any polarized abelian variety $(X,L)$ of type $(1,2^{(g-1)})$ has basepoints by \cite[Corollary 2.6]{MR1360615}. 
Hence we have $\beta (1,2^{(g-1)})=1$ and $\beta(2,4^{(g-1)}) =1/2$.
Thus the bound in $\chi(D) > 2^{2g-1}$ is sharp.
\end{rem}

The strategy of the proof is the same as that of  \autoref{prop_refiniment}.
That is, we decompose $D$ to types whose basepoint-freeness thresholds are already known to be less than $1/2$.
By \autoref{thm_Ito:2020ab}, \autoref{ex_g=3_gen_type_variant}, \autoref{cor_111222d},
we know that
\begin{itemize}
\item $\beta(1^{(g-1)},d) <1/2$ if $d \geq 2^{g+1} -1$.
\item More generally, $\beta(1^{(n-1)}, 2^{(n')} ,d ) < 1/2$ for $n \geq 1, n' \geq 0$ if $d \geq 2^{n+1} +2n' -1 $.
\item $\beta(1,c,c) < 1/2$ for $c \geq 5$.
\end{itemize}

%If $d_1 \geq 2$, we use the following lemma.

We use the following lemmas in the proof of \autoref{thm_beta<1/2}.
 
\begin{lem}\label{lem_d_1=2}
Let $D =(d_1, \dots,d_g)$ be a type such that $d_1 \geq 3$ or $d_1=2$ and $|D| \geq 4g$.
Then it holds that $\beta(D)<1/2$.
\end{lem}

\begin{proof}
Set $D'= (d_1/d_1, d_2/d_1,\dots, d_g/d_1) $.
If $d_1 \geq 3$, we have $\beta(D) \leq \beta(D')/d_1 \leq 1/d_1 <1/2$.

If $d_1 =2$,
$\beta(D)<1/2$ is equivalent to $\beta(D') <1$,
which follows from $|D'| =|D|/2 \geq  2g$ and \cite[Theorem 1.1]{MR2231163}.
\end{proof}

\begin{lem}\label{lem_reduction}
To prove \autoref{thm_beta<1/2},
we may assume $d_1 =1$, $g \geq 3$ and $d_g \geq 5$. 
\end{lem}

\begin{proof}
By \autoref{lem_d_1=2}, we may assume $d_1 \leq 2$.

Assume $d_1=2$.
Then $\chi(D) \geq 2^{2g-1 } $ is equivalent to $d_2 \cdots d_g \geq 2^{2g-2}$.
Hence we have $(d_2 + \dots +d_{g})/(g-1) \geq \sqrt[g-1]{2^{2g-2}} =4$.
Furthermore, the equality holds if and only if $ d_2 \cdots d_g = 2^{2g-2} $ and $d_2 = \dots =d_g$, that is, $d_2 =\dots =d_g=4$. 
Hence if $\chi(D) \geq 2^{2g-1 } $ and $D \neq (2,4^{(g-1)})$, then it holds that $d_2 + \dots +d_{g} > 4 (g-1 ) =4g-4$.
Since $ d_2, \dots,d_g$ are  even by $d_1=2$, 
we have $d_2 + \dots +d_{g} \geq 4g-2$ and hence $|D| \geq 4g$.
Thus $\beta(D) <1/2$  follows from \autoref{lem_d_1=2}.
Hence we may assume $d_1=1$.

If $g=1$,  $\chi(D) \geq 2^{2g-1}$ and $D \neq (2,4^{(g-1)}) =(2)$ mean that $d_1 \geq 3$.
Hence $\beta(D) = 1/d_1 <1/2$ follows.

Assume $g=2$.
By \autoref{thm_Ito:2020ab}, \autoref{lem_d_1=2} and \cite[Table 1]{Ito:2020ab},
 $\beta(D) \geq 1/2$ holds if  and only if
$D$ is one of the following:
\[
(1,1), (1,2), (1,3) , (1,4) ,(1,5), (1,6) , (2,2), (2,4).
\]
Since $\chi(D) < 8=2^{2g-1}$ for these types except for $(2, 4^{(g-1)}) = (2,4)$,
 \autoref{thm_beta<1/2} holds for $g=2$.
Hence we may assume $g \geq 3$.

As we already see, we may assume $d_1=1$.
Then $2^{2g-1} \leq \chi(D) = d_2 \cdots d_g \leq d_g^{g-1} $, which implies $d_g >4$.
Hence we may assume that $d_g \geq 5$.
\end{proof}

%In induction process, we use the following lemma.
%We use  the following lemma in the proof of \autoref{thm_beta<1/2}.

\begin{lem}\label{lem_decomp_4^g}
Assume that \autoref{thm_beta<1/2} holds when $\lambda(D) \leq g-1$.
Let $D $ be a type with $\lambda(D)=g$, $\chi(D) \geq 2^{2g-1}$. 
Then $\beta(D) <1/2$ holds if there exists a decomposition $D=D_1 \times D_2$ with $g_i=\lambda(D_i) $ 
which satisfies
\[
g_1 \geq 1, \quad \chi(D_1)  \leq 2^{2g_1}, \quad \beta(D_1 ) <1/2, \quad D_2 \neq (2, 4^{(g_2-1)}).
\]
Furthermore, the last condition $D_2 \neq (2, 4^{(g_2-1)})$ is satisfied if $\chi(D_1) < 2^{2g_1}$.
\end{lem}

\begin{proof}
By 
$
2^{2g-1} \leq  \chi(D) =\chi(D_1) \cdot \chi(D_2) \leq 2^{2 g_1} \cdot  \chi(D_2),
$
we have $\chi(D_2) \geq 2^{2g-1 -2 g_1} = 2^{2 g_2-1}  $ since $g=g_1+g_2$.
By $g_2 =g-g_1 \leq g-1$ and $D_2 \neq (2, 4^{(g_2-1)})$, we have $\beta(D_2) < 1/2$ by assumption.
Thus $\beta(D) =\beta(D_1 \times D_2) < 1/2$ follows from  \autoref{prop_beta(d_1,...,d_g)}.

Since $\chi(D_1) < 2^{2g_1}$ implies $\chi(D_2) > 2^{2 g_2-1} = \chi(2, 4^{(g_2-1)})$,
 the last statement follows.
\end{proof}

\begin{proof}[Proof of \autoref{thm_beta<1/2}]
The proof is done by the induction on $\lambda(D)$,
that is,
we assume that \autoref{thm_beta<1/2} holds when $\lambda(D) \leq g-1$,
and prove  \autoref{thm_beta<1/2} when $\lambda(D) =g$.
By \autoref{lem_reduction}
we may assume 
$d_1 =1, g \geq 3,   d_g \geq 5$.

\vspace{2mm}
Let $D=(d_1,\dots,d_g)$ be a type such that $\chi(D) \geq 2^{2g-1}$ and $D \neq (2,4^{(g-1)})$.
Set 
\[
n=\min\{ 1\leq  i \leq g  \, | \, d_i \geq 2 \} \geq 1, \quad n'=\#\{  1 \leq  i \leq g  \, | \, d_i=2 \} \geq 0.
\]
Hence we can write 
\[
D=(1^{(n-1) } , 2^{(n')}, d_{n+n'} , \dots, d_g)
\]
with $d_{n+n'}  \geq 3$.

\vspace{2mm}
\noindent
{\bf Case 1}: $d_g \leq 2n'+4$.

\vspace{1mm}
In this case, we have $n' \geq 1$ by $d_g \geq 5$.
In particular, $d_g$ is even and $d_g \geq 6$.
Hence $l:=(d_g-4)/2 $ is a positive integer.
By $d_g \leq 2n'+4$, we have $l \leq n'$.
Consider a decomposition $D=D_1 \times D_2 = (2^{(l)}, d_g) \times  (1^{(n-1)}, 2^{(n'-l)}, d_{n+n'}, \dots, d_{g-1}).$
Since $|D_1| = 2l +d_g = 4l+4 =4 \lambda(D_1)$, we have $\beta(D_1) < 1/2$ by \autoref{lem_d_1=2}.
Thus 
to show $\beta(D) <1/2$,
it suffices to see $\chi(D_1) < 2^{2\lambda(D_1)}$ by \autoref{lem_decomp_4^g}.

By $ d_g=2l+4$, $\lambda(D_1) =l+1$ and $ \chi(D_1) =2^l  \cdot d_g$,
the inequality $\chi(D_1) < 2^{2\lambda(D_1)}$ is equivalent to $l+2 < 2^{l+1}$,
which holds by $l \geq 1$.

\vspace{2mm}
\noindent
{\bf Case 2}: $d_g \geq 2n'+5$.

\vspace{1mm}
Let $m \geq 1$ be the integer such that $2^{m+1} -1 \leq d_g-2n'  < 2^{m+2} -1$.

If $m \geq n$, we consider a decomposition $D=D_1 \times D_2
=(1^{(n-1)}, 2^{(n')}, d_g) \times (d_{n+n'},\dots,d_{g-1})$.
By the choice of $m$ and $m \geq n$,
we have $\beta(D_1) < 1/2$ by \autoref{cor_111222d}.
Since $\beta(D_2) \leq 1/d_{n+n'} < 1/2 $, we have $\beta(D) < 1/2$.

If $m \leq n-1$, 
we consider  $D=D_1 \times D_2
=(1^{(m-1)} , 2^{(n')} ,d_g) \times (1^{(n-m)}, d_{n+n'},\dots,d_{g-1})$.
By the choice of $m$,
we have $\beta(D_1) < 1/2$ by \autoref{cor_111222d}.
Since $D_2 \neq (2, 4^{(\lambda(D_2)-1)})$ by $n-m \geq 1$,
$\beta(D) < 1/2$ holds if $\chi(D_1) \leq  2^{2\lambda(D_1)}$ by \autoref{lem_decomp_4^g}.
Since $\chi(D_1) = 2^{n'} \cdot d_g$ and $\lambda(D_1) = m+n'$,
$\chi(D_1) \leq 2^{2\lambda(D_1)}$ is equivalent to $ d_g \leq 2^{2m+n'}$.
Thus we have $\beta(D) < 1/2$ if $ 2^{m+2} +2n'-2 \leq 2^{2m+n'}$
since $d_g \leq 2^{m+2} +2n'-2$ by the choice of $m$.

\begin{claim}\label{calim_inequality}
If $(m,n') \neq (1,0)$,
$\beta(D) < 1/2$ holds.
\end{claim}

\begin{proof}[Proof of \autoref{calim_inequality}]
It suffices to see $ 2^{m+2} +2n'-2 \leq 2^{2m+n'}$.
Recall that $m \geq 1$ and $n' \geq 0$.
If $(m,n') \neq (1,0)$, then $m+n' \geq 2$ and hence it holds that 
\begin{align*}
2^{2m+n'} - 2^{m+2} &=2^{m+2} (2^{m+n'-2} -1) \\
& \geq 2^{2} (2^{m+n'-2} -1) \\
&\geq 2^2 (2^{n'-1} -1) = 2 \cdot 2^{n'} -4 \geq  2 (n'+1) -4 =2n'-2.
\end{align*}
\end{proof}

Thus we may assume $ (m,n') =(1,0)$.
Then $5 \leq d_g < 2^{m+2} +2n'-1 =7$, that is, $d_g =5,6$.

If $d_g=5$, we have $D=(1^{(a)}, 5^{(b)})$ with $a =n-1, a+b=g$.
Then 
\[
2^{2(a+b)-1} = 2^{2g-1}  \leq \chi(D) = 5^{b} < 8^b =2^{3b}
\]
implies $3b  > 2a+2b -1$, that is,  $b \geq 2a$.
Hence $D=(1^{(a)}, 5^{(b)})$ is decomposed as 
\[
(1,5,5) \times \dots \times (1,5,5) \times (5^{(b-2a)}),
\]
where $(1,5,5)$ appears $a$ times.
Since $\beta(1,5,5) <1/2$ by \autoref{ex_g=3_gen_type_variant} and $\beta(5^{(b-2a)}) =1/5 $,
we have $\beta(D) < 1/2$.

If $d_g =6$,
$D=(1^{(n-1)} , 3^{(a)}, 6^{(b)})$ for some $a \geq 0, b \geq 1$ with $n-1+a+b=g$ since $d_i \neq 2$ for any $i$ by $n'=0$.
Then 
\[
2^{2g-1} \leq \chi(D) = 3^a \cdot 6^{b} < 4^a \cdot 8^b =2^{2a+3b}
\]
implies $2a + 3b > 2g-1 = 2a + 2b +2(n-1) -1$,
that is, $b \geq 2(n-1)$.
Hence $D=(1^{(n-1)} , 3^{(a)}, 6^{(b)})$ is decomposed as 
\[
(1,6,6) \times \dots \times (1,6,6) \times (3^{(a)}, 6^{(b-2(n-1))}),
\]
where $(1,6,6)$ appears $n-1$ times.
Since $\beta(1,6,6) < 1/2 $ by  \autoref{ex_g=3_gen_type_variant} and $\beta(3^{(a)}, 6^{(b-2(n-1))}) \leq 1/3$,
we have $\beta(D) < 1/2$.
Hence \autoref{thm_beta<1/2} is proved.
\end{proof}

\begin{proof}[Proof of \autoref{thm_proj_normality}]
By \autoref{thm_Bpf_threshold} (2) and \autoref{rem_beta(D)} (i),
\autoref{thm_proj_normality} is an immediate consequence of \autoref{thm_beta<1/2}.
\end{proof}

\bibliographystyle{amsalpha}
%\bibliography{mainbibs}
\providecommand{\bysame}{\leavevmode\hbox to3em{\hrulefill}\thinspace}
\providecommand{\MR}{\relax\ifhmode\unskip\space\fi MR }
\providecommand{\MRhref}[2]{%
  \href{http://www.ams.org/mathscinet-getitem?mr=#1}{#2}
}
\providecommand{\href}[2]{#2}

\end{document}